\documentclass[11pt,twoside]{article}
\usepackage{amssymb}
\usepackage{latexsym}
\usepackage{graphicx}
\usepackage{amsmath}
\usepackage{hyperref,amsthm}
\usepackage{epsfig}

\numberwithin{equation}{section}
\normalsize \makeatletter
\newtheorem{thm}{Theorem}[section]
\newtheorem{lemma}[thm]{Lemma}

\newtheorem{definition}[thm]{Definition}

\newtheorem{rem}[thm]{Remark}

\newcommand{\be}{\begin{equation}}
\newcommand{\ee}{\end{equation}}
\newcommand{\bea}{\begin{eqnarray*}}
\newcommand{\eea}{\end{eqnarray*}}
\newcommand{\mR}{\mathbb{R}}

\newcommand\argmin{\operatorname{arg\,min}}

\newenvironment{keyword}{\smallskip\noindent{\bf Keywords.}
                          \hskip\labelsep}{}

\begin{document}
\date{}
\title{ \textbf {Compressed Sensing with coherent tight frames via $l_q$-minimization for $0<q\leq1$
\footnote{This work is supported by NSF of China under grant numbers
10771190, 10971189 and Zhejiang Provincial NSF of China under grant
number Y6090091.}}}
\author{Song Li and Junhong Lin\thanks{Corresponding author: Junhong Lin.
\newline E-mail adress: songli@zju.eud.cn (S. Li), jhlin5@hotmail.com (J. Lin).
\newline 2010 Mathematics Subiect Classification. Primary 94A12, 94A15, 94A08, 68P30; Secondary 41A63, 15B52, 42C15.}
\\Department of Mathematics, Zhejiang University \\Hangzhou, 310027, P. R. China } \maketitle \baselineskip 15pt

\begin{abstract}
Our aim of this article is to reconstruct a signal from undersampled
data in the situation that the signal is sparse in terms of a tight
frame. We present a condition, which is independent of the coherence
of the tight frame, to guarantee accurate recovery of signals which
are sparse in the tight frame, from undersampled data with minimal
$l_1$-norm of transform coefficients. This improves the result in
\cite{Candes}. Also, the $l_q$-minimization $(0<q<1)$ approaches are
introduced. We show that under a suitable condition, there exists a
value $q_0\in(0,1]$ such that for any $q\in(0,q_0)$, each solution
of the $l_q$-minimization is approximately well to the true signal.
In particular, when the tight frame is an identity matrix or an
orthonormal basis, all results obtained in this paper appeared in
\cite{Mo} and \cite{Lai}.

\end{abstract}

\section{Introduction}
Compressed sensing is a new type of sampling theory, that predicts
sparse signals can be reconstructed from what was previously
believed to be incomplete information \cite{candes1,candes2,donoho}.
By now, applications of compressed sensing are abundant and range
from medical imaging and error correction to radar and remote
sensing, see \cite{a1,a2} and the references therein.

In compressed sensing, one considers the following model: \be
\label{e7}y=Ax+z,\ee where $A$ is a known $m\times n$ measurement
matrix (with $m\ll n$) and $z\in \mR^n$ is a vector of measurement
errors. The goal is to reconstruct the unknown signal $x$ based on
$y$ and $A$. The key idea of compressed sensing relies on that
signal is sparse or approximately sparse. A naive approach for
solving this problem consists in searching for the sparsest vector
that is consistent with the linear measurements, which leads to:
$$\min\limits_{\tilde{x}\in{\mR^n}}\|\tilde{x}\|_0\quad\mbox{subject
to}\quad \|A\tilde{x}-y\|_2\leq \varepsilon,
\eqno{(L_{0,\varepsilon})}$$where $\|x\|_0$ is the numbers of
nonzero components of $x=(x_1,...,x_n)\in\mR^n$, $\|\cdot\|_2$
denotes the standard Euclidean norm and $\varepsilon\geq 0$ is a
likely upper bound on the noise level $\|z\|_2$. If $\varepsilon=0$,
it is for the noiseless case. If $\varepsilon>0$, it is for the
noisy case. We call that a vector $x$ is $s$-sparse if $\|x\|_0\leq
s$. Unfortunately, solving $(L_{0,\varepsilon})$ directly is NP-hard
in general and thus is computationally infeasible
\cite{mallat,natarajan}. One of the practical and tractable
alternatives to $(L_{0,\varepsilon})$ proposed in the literature is:
$$\min\limits_{\tilde{x}\in{\mR^n}}\|\tilde{x}\|_1\quad\mbox{subject
to}\quad \|A\tilde{x}-y\|_2\leq \varepsilon,
\eqno{(L_{1,\varepsilon})}$$ which is a convex optimization problem
and can be seen as a convex relaxation of $(L_{0,\varepsilon})$. The
restricted isometry property (RIP), which first appeared in
\cite{candes3}, is one of the most commonly used frameworks for
sparse recovery via $(L_{1,\varepsilon})$. For an integer $s$ with
$1\leq s\leq n$, we define the $s$-restricted isometry constants of
a matrix $A$ as the smallest constants satisfying
$$(1-\delta_s)\|x\|_2^2\leq\|Ax\|_2^2\leq(1+\delta_s)\|x\|_2^2$$
for all $s$-sparse vectors $x$ in $\mR^n$. By computation, one would
observe that
\be\label{e(delta)}\delta_s=\max_{T\subset\{1,\cdots,n\},|T|\leq
s}\|A_T^*A_T-I\|,\ee where $\|\cdot\|$ denotes the spectral norm of
a matrix. However, it would be computationally difficult to compute
$\delta_s$ using (\ref{e(delta)}). One of the good news is that many
types of random measurement matrices have small restricted isometry
constants with very high probability provided that the measurements
$m$ is large enough \cite{candes2,Rudelson,Baraniuk}. Since
$\delta_{2s}<1$ is a necessary and sufficient condition to guarantee
that any $s$-sparse vector $f$ is exactly recovered via ($L_{0,0}$)
in the noiseless case, many attentions have been focused on
$\delta_{2s}$ in the literature\cite{candes4,Foucart,cai2,Mo}.
Cand\`{e}s \cite{candes4} showed that under the condition
$\delta_{2s}<0.414$, one can recover a (approximately) sparse signal
with a small or zero error using $(L_{1,\varepsilon})$. Later, the
sufficient condition on $\delta_{2s}$ were improved to
$\delta_{2s}<0.453$ by Lai et al. \cite{Foucart} and
$\delta_{2s}<0.472$ by Cai et al. \cite{cai2}, respectively.
Recently, Li and Mo \cite{Mo} has improved the sufficient condition
to $\delta_{2s}<0.493$ and for some special cases, to
$\delta_{2s}<0.656$. To the best of our knowledge, this is the best
known bound on $\delta_{2s}$ in the literature. On the other hand,
Davies and Gribonval \cite{Davies} constructed examples which showed
that if $\delta_{2s}\geq 1/\sqrt{2}$, exact recovery of certain
$s$-sparse signals can fail in the noiseless case.

For signals which are sparse in the standard coordinate basis or
sparse in terms of some other orthonormal basis, the mechanism above
holds. However, in practical examples, there are numerous signals of
interest which are not sparse in an orthonormal basis. More often
than not, sparsity is not expressed in terms of an orthogonal basis
but in terms of an overcomplete dictionary \cite{Candes}.

In this paper, we consider recovery of signals that are sparse in
terms of a tight frame from undersampled data. Formally, let $D$ be
a $n\times d$ matrix whose $d$ columns $D_1,...,D_d$ form a tight
frame for $\mR^n$, i.e.
$$f=\sum_k\langle f,D_k\rangle D_k\quad\mbox{and}\quad
\|f\|_2^2=\sum_k|\langle f,D_k\rangle|^2\quad \mbox{for all}\quad
f\in\mR^n,$$ where $\langle\cdot,\cdot\rangle$ denotes the standard
Euclidean inner product. Our object in this paper is to reconstruct
the unknown signal $f\in\mR^n$ from a collection of $m$ linear
measurements $y=Af+z$ under the assumption that $D^*f$ is sparse or
nearly sparse. Such problem has been considered in
\cite{tropp,Bajwa,Rauhut,Candes}. The methods in
\cite{tropp,Rauhut,Bajwa} force incoherence on the dictionary $D$ so
that the matrix $AD$ conforms to the above standard compressed
sensing results. As a result, they may not be suitable for
dictionary which are largely correlated. One new alternative way
imposing no such properties on the dictionary $D$ for reconstructing
the signal $f$ from $y=Af+z$ is to find the solution of
$l_1$-minimization:
$$\hat{f}=
\argmin\limits_{\tilde{f}\in{\mR^n}}\|D^*\tilde{f}\|_1\quad\mbox{subject
to}\quad \|A\tilde{f}-y\|_2\leq \varepsilon,
\eqno{(P_{1,\varepsilon})}$$ where again $\varepsilon\geq 0$ is a
likely upper bound on the noise level $\|z\|_2$. For discussion of
the performance of this method, we would like to introduce the
definition of $D$-RIP of a measurement matrix, which first appeared
in \cite{Candes} and is a natural extension to the standard RIP.
\begin{definition}[$D$-RIP] Let $D$ be a tight frame and $\Sigma_s$ be the set of all $s$-sparse
vectors in $\mathbb{R}^d$. A measurement matrix $A$ is said to obey
the restricted isometry property adapted to $D$ (abbreviated
$D$-RIP) with constants $\delta_s$ if
$$(1-\delta_s)\|Dv\|_2^2\leq\|ADv\|_2^2\leq(1+\delta_s)\|Dv\|_2^2$$
holds for all $v\in\Sigma_s$.
\end{definition}For the rest of this paper, $D$ is a $n\times d$ tight frame and $\delta_s$ denotes the
$D$-RIP constants with order $s$ of the measurement matrix $A$
without special mentioning. Throughout this paper, denote $x_{[s]}$
to be the vector consisting of the $s$-largest coefficients of
$v\in\mR^d$ in magnitude:
$$x_{[s]}=\argmin_{\|\tilde{x}\|_0\leq s}\|x-\tilde{x}\|_2.$$
Cand\`{e}s et al. \cite{Candes} showed that if $A$ satisfies $D$-RIP
with $\delta_{2s}<0.08$ (in fact, a weaker condition
$\delta_{7s}<0.6$ was stated), then the solution $\hat{f}$ to
$(P_{1,\varepsilon})$ satisfies \be\label{candes
equality}\|\hat{f}-f\|_2\leq
C_0\frac{\|D^*f-(D^*f)_{[s]}\|_1}{\sqrt{s}}+C_1\varepsilon,\ee where
the constants $C_0$ and $C_1$ may only depend on $\delta_{2s}$. It
is easy to see that it's computationally difficult to verify the
$D$-RIP for a given deterministic matrix. But for matrices with
Gaussian, subgaussian, or Bernoulli entries, the $D$-RIP condition
will be satisfied with overwhelming probability provided that the
numbers of measurements $m$ is on the order of $s\log(d/s)$. In
fact, for any $m\times n$ matrix $A$ obeying for any fixed
$\nu\in\mR^n$,\be\label{e(concentrated)}\mathbb{P}\left(\big|\|A\nu\|_2^2-\|\nu\|_2^2\big|\geq\delta\|\nu\|_2^2\right)\leq
c{e}^{-\gamma m\delta^2},\quad \delta\in(0,1)\ee($\gamma$, $c$ are
positive numerical constants) will satisfy the $D$-RIP with
overwhelming probability provided that $m \gtrsim s\log(d/s)$
\cite{Candes}. Therefore, by using $D$-RIP, the work is independent
on the coherence of the dictionary. The result holds even when the
coherence of the dictionary $D$ is maximal, meaning two columns are
completely correlated. Although Can\`{e}s et al. in \cite{Candes}
gave the sufficient condition on $\delta_{2s}$ to guarantee
approximately recovery of a signal via $(P_{1,\varepsilon})$, the
bound on $\delta_{2s}$ is much weaker comparing to the case for
which $D$ is an orthonormal basis. We focus on improving it in this
paper.

Our first goal of this paper is to show that the sufficient
condition on $\delta_{2s}$ above can be improved to
$\delta_{2s}<0.493$. And in some special cases, the sufficient
condition can be improved to $\delta_{2s}<0.656$. These results are
given and proved in Section 3. Weakening the $D$-RIP condition has
several benefits. First, it allows more measurement matrices to be
used in compressed sensing. Secondly, it give better error
estimation in a general problem to recover noisy compressible
signal. For example, if $\delta_{2s}=1/14$, Then by \cite[Corollary
3.4]{needell}, $\delta_{7s}\leq 0.5$. Using the approach in
\cite{Candes} one would get a estimation in (\ref{candes equality})
with $C_0=30, C_1=62$. While by Theorem \ref{thm5} in Section 3 of
this paper, one would get a estimation in (\ref{candes equality})
with $C_0\simeq5.06$ and $C_1\simeq10.57$. Finally, for the same
measurement random matrix $A$ which satisfies
(\ref{e(concentrated)}), a standard argument as in
\cite{Baraniuk,Candes} shows that it allows recovering a sparse
signal with more non-zero transform coefficients. In a nutshell,
weakening the $D$-RIP condition for $(P_{1,\varepsilon})$ is as
important as weakening the RIP condition for classical
$(L_{1,\varepsilon})$.

Note that the $l_0$-norm is the limit as $q\rightarrow 0$ of the
$l_q$-norm in the following
sense:\be\|x\|_0=\lim_{q\rightarrow0}\|x\|_q^q=\lim_{q\rightarrow0}\sum_{j}|x_j|^q\nonumber.\ee
Thus $l_q$-norm with $0<q<1$ can be used for measuring sparsity.
Therefore, one alternative way of finding the solution of
$(L_{0,\varepsilon})$ proposed in the literature is to solve:
$$\min_{\tilde{x}\in{\mR^n}}\|\tilde{x}\|_q\quad\mbox{subject
to}\quad \|A\tilde{x}-y\|_2\leq \varepsilon.
\eqno{(L_{q,\varepsilon})}$$ This is a non-convex optimization
problem since $l_q$-norm with $0 <q< 1$ is not a norm but a
quasi-norm. For any fixed $0 < q < 1$, while checking the global
minimal value of $(L_{q,\varepsilon})$ is NP-hard, computing a local
minimizer of the problem is polynomial time doable \cite{ge}.
Therefore, to solve $(L_{q,\varepsilon})$ is still much faster than
to solve $(L_{0,\varepsilon})$ at least locally. Reconstruction
sparse signals via $(L_{q,\varepsilon})$ with $0<q<1$ have been
considered in the literature in a series of papers
\cite{chartrand2,Chartrand,Saab,Foucart,daubechies,Lai} and some of
the virtues are highlighted recently. Lai and Liu \cite{Lai} showed
that as long as the classical restricted isometry constant
$\delta_{2s}<1/2$, there exist a value $q_0\in (0,1]$ such that for
any $q \in(0, q_0)$, each solution of ${(L_{q,0})}$ for the sparse
solution of any underdetermined linear system is the sparsest
solution. Thus, it's natural for us to consider the reconstruction
of a signal $f$ from $y=Af+z$ by the method of $l_q$-minimization
($0<q<1$):
$$\hat{f}=
\argmin_{\tilde{f}\in{\mR^n}}\|D^*\tilde{f}\|_q\quad\mbox{subject
to}\quad \|A\tilde{f}-y\|_2\leq \varepsilon.
\eqno{(P_{q,\varepsilon})}$$ Our second goal of this paper is to
estimate the approximately error between $\hat{f}$ and $f$ when
using the $l_q$-minimization $(P_{q,\varepsilon})$. We show that if
the measurement matrix satisfies the $D$-RIP condition with
$\delta_{2s}<1/2$, then there exists a value
$q_0=q_0(\delta_{2s})\in(0,1]$ such that for any $q\in(0,q_0),$ each
solution of the $l_q$-minimization is approximately well to the true
signal $f$.

This paper is organized as follows. Some lemmas and notations are
introduced in Section 2. Section 3 is devoted to discuss recovery of
a signal from noisy data via $l_1$-minimization. We begin by give
some lemmas in this section. And then we discuss approximately
recovery of a signal for the general case in Subsection 3.1 while
for the special case in Subsection 3.2. Our main results in this
section are Theorem \ref{thm5}, Theorem \ref{thm6}. In Section 4, we
discuss recovery of a signal from noisy data via $l_q$-minimization
with $0<q\leq1$. Some lemmas and notations are introduced at the
beginning. Subsequently, we give the main result Theorem \ref{thm7}
and prove it.

\section{Lemmas}
We will give some lemmas and notations first. We begin by discussing
some of the results of recovery of a signal by $l_0$-minimization:
$$\hat{f}=\argmin\limits_{\tilde{f}\in\mR^n}\|D^*\tilde{f}\|_0,\quad\mbox{subject
to }\quad A\tilde{f}=y.\eqno{(P_0)}$$

\begin{lemma}\label{p0}If $\delta_{2s}<1$, then every signal $f$ such that $\|D^*f\|_0\leq s$ can be exactly
recovered by solving $(P_0)$. \end{lemma}
\begin{proof}Since $\hat{f}$ is the solution of ($P_0$), we have $\|D^*\hat{f}\|_0\leq \|D^*{f}\|_0\leq
s$. Then $h=\hat{f}-f$ satisfies $\|D^*{h}\|_0\leq2s$. Notice that
$0=Ah=ADD^*h$. Thus according to the definition of $D$-RIP, we have
$$(1-\delta_{2s})\|DD^*h\|_2^2\leq\|ADD^*h\|_2^2=0.$$ Combining with
$\delta_{2s}<1,$ we get $0=DD^*h=h$. The proof is finished.
\end{proof}

\begin{lemma}\label{lemma2}For all $u,v\in\Sigma_s$, we have
$$\langle ADu,ADv\rangle \leq \delta_{2s}\|Du\|_2\|Dv\|_2+\ \langle Du,Dv\rangle.$$
\end{lemma}
\begin{proof}For $u,v\in\Sigma_s$, assume that
$\|Du\|_2=\|Dv\|_2=1$. By the definition of $D$-RIP, we have
\bea &&\langle ADu,ADv\rangle=\frac{1}{4}\left\{\|ADu+ADv\|_2^2-\|ADu-ADv\|_2^2\right\}\\
&\leq&
\frac{1}{4}\left\{(1+\delta_{2s})\|Du+Dv\|_2^2-(1-\delta_{2s})\|Du-Dv\|_2^2\right\}
=\delta_{2s}+\ \langle Du,Dv\rangle.\eea Thus by a simple
modification, we conclude the proof.
\end{proof}

\begin{rem}Notice that when $D$ is an identity matrix and for any $s$-sparse vectors $u,v$ with
disjoint supports, one have $$\langle Au,Av\rangle \leq
\delta_{2s}\|u\|_2\|v\|_2,$$ which is the same as \cite[Lemma
2.1]{candes4}.
\end{rem}

For $T\subset\{1,...,d\},$ denote by $D_T$ the matrix $D$ restricted
to the columns indexed by $T$, and write $D_T^*$ to mean $(D_T)^*$,
$T^c$ to mean the complement of $T$ in $\{1,\cdots,d\}$. Given a
vector $h\in\mR^n$, we write
$D^*h=(x_1,\cdots,x_s,\cdots,x_{2s},\cdots,x_d).$ Let
$D^*h=D_{T_0}^*h+D_{T_1}^*h+\cdots D_{T_l}^*h,$ where
$D_{T_0}^*h=(x_1,x_2,\cdots,x_s,0,0,\cdots,0),$
$D_{T_1}^*h=(0,0,\cdots,0,x_{s+1},x_{s+2},\cdots,x_{2s},0,0,\cdots,0),$
$D_{T_2}^*h=(0,0,\cdots,0,x_{2s+1},x_{2s+2},\cdots,x_{3s},0,0,$
$\cdots,0),$$\cdots$ and $D_{T_l}^*h=(0,0,\cdots,0,x_{ls+1},$
$x_{ls+2},\cdots,x_{d}).$ Denote $T_{01}=T_0\cup T_1$ and $T=T_0$.
For simplicity, we assume that $\delta_{2s}<1$.

\begin{lemma}\label{lemma3}We have
$$\|\sum\limits_{j=2}^{l}ADD_{T_j}^*h\|_2^2\leq\sum\limits_{j=2}^{l}\|D_{T_j}^*h\|_2^2
+\delta_{2s}\left(\sum\limits_{j=2}^{l}\|D_{T_j}^*h\|_2\right)^2.$$
\end{lemma}
\begin{proof}By the definition of $\delta_{2s}$ and Lemma
\ref{lemma2}, we have
\bea&&\|\sum\limits_{j=2}^{l}ADD_{T_j}^*h\|_2^2=\sum\limits_{i,j=2}^{l}\langle ADD_{T_i}^*h,ADD_{T_j}^*h\rangle\\
&=&\sum\limits_{j=2}^{l}\|ADD_{T_j}^*h\|_2^2+2\sum\limits_{2\leq
i<j\leq l}\langle ADD_{T_i}^*h,ADD_{T_j}^*h\rangle\\
&\leq&(1+\delta_{2s})\sum\limits_{j=2}^{l}\|DD_{T_j}^*h\|_2^2+2\sum\limits_{2\leq
i<j\leq l}\langle
DD_{T_i}^*h,DD_{T_j}^*h\rangle+2\delta_{2s}\sum\limits_{2\leq
i<j\leq l}\|DD_{T_{i}}^*h\|_2\|DD_{T_{j}}^*h\|_2\\
&=&\|\sum\limits_{j=2}^{l}DD_{T_j}^*h\|_2^2+\delta_{2s}\left(\sum\limits_{j=2}^{l}\|DD_{T_j}^*h\|_2\right)^2.\eea

Notice that by $\|h\|_2^2=\sum\limits_{j=0}^{l}\|D_{T_j}^*h\|_2^2$,
\begin{eqnarray}\label{e9}&&\|\sum\limits_{j=2}^{l}DD_{T_j}^*h\|_2^2=\|h-DD_{T_{01}}^*h\|_2^2
=\|h\|_2^2-2\langle
h,DD_{T_{01}}^*h\rangle+\|DD_{T_{01}}^*h\|_2^2\nonumber\\
&=&\|h\|_2^2-2\|D_{T_{01}}^*h\|^2+\|DD_{T_{01}}^*h\|_2^2=\sum\limits_{j=2}^{l}\|D_{T_j}^*h\|_2^2-\|D_{T_{01}}^*h\|_2^2
+\|DD_{T_{01}}^*h\|_2^2.\end{eqnarray}
Therefore, we get\begin{eqnarray}\label{e10}
\|\sum\limits_{j=2}^{l}ADD_{T_j}^*h\|_2^2&\leq&
\sum\limits_{j=2}^{l}\|D_{T_j}^*h\|_2^2-\|D_{T_{01}}^*h\|_2^2+\|DD_{T_{01}}^*h\|_2^2
+\delta_{2s}\left(\sum\limits_{j=2}^{l}\|D_{T_j}^*h\|_2\right)^2\\
&\leq&\sum\limits_{j=2}^{l}\|D_{T_j}^*h\|_2^2
+\delta_{2s}\left(\sum\limits_{j=2}^{l}\|D_{T_j}^*h\|_2\right)^2,\nonumber\end{eqnarray}
where we have used
$\|DD_{T_{j}}^*h\|_2\leq\|D_{T_{j}}^*h\|_2,j\in\{0,1,\cdots l\}.$
\end{proof}

\begin{lemma}\label{lemma4}We have
$$\|\sum\limits_{j=2}^{l}ADD_{T_j}^*h\|_2^2-\|ADD_{T_{01}}^*h\|_2^2\leq\sum\limits_{j=2}^{l}\|D_{T_j}^*h\|_2^2
+\delta_{2s}\left(\sum\limits_{j=2}^{l}\|D_{T_j}^*h\|_2\right)^2-(1-\delta_{2s})\|D_{T_{01}}^*h\|_2^2.$$
\end{lemma}
\begin{proof}By the definition of $\delta_{2s}$ and (\ref{e10}), we have
\bea&&\|\sum\limits_{j=2}^{l}ADD_{T_j}^*h\|_2^2-\|ADD_{T_{01}}^*h\|_2^2\\
&\leq&\|\sum\limits_{j=2}^{l}ADD_{T_j}^*h\|_2^2-(1-\delta_{2s})\|DD_{T_{01}}^*h\|_2^2\\
&\leq&\sum\limits_{j=2}^{l}\|D_{T_j}^*h\|_2^2
+\delta_{2s}\left(\sum\limits_{j=2}^{l}\|D_{T_j}^*h\|_2\right)^2+\delta_{2s}\|DD_{T_{01}}^*h\|_2^2-\|D_{T_{01}}^*h\|_2^2\\
&\leq&\sum\limits_{j=2}^{l}\|D_{T_j}^*h\|_2^2
+\delta_{2s}\left(\sum\limits_{j=2}^{l}\|D_{T_j}^*h\|_2\right)^2-(1-\delta_{2s})\|D_{T_{01}}^*h\|_2^2.\eea
\end{proof}

\section{Recovery via $l_1$-minimization}
In this section, we are concerned with the reconstruction of a
signal $f$ from $y=Af+z$ by the method of $l_1$-minimization:
$(P_{1,\varepsilon})$.

Let $h=\hat{f}-f$, where $\hat{f}$ is the solution of
$(P_{1,\varepsilon})$ and $f$ is the original signal. We use the
same assumptions as in Section 2. Furthermore, rearranging the
indices if necessary, we assume that the first $s$ coordinates of
$D^*f$ are the largest in magnitude and
$|x_{s+1}|\geq|x|_{s+2}\geq\cdots\geq|x_d|.$ For the rest of this
section, we will always assume that
$\|D_{T_1}^*h\|_1=\omega\sum_{j=1}^{l}\|D_{T_j}^*h\|_1$ for some
nonnegative number $\omega\in[0,1]$. Then we have
$\sum_{j=2}^{l}\|D_{T_j}^*h\|_1=(1-\omega)\sum_{j=1}^{l}\|D_{T_j}^*h\|_1$.

\begin{lemma}\label{lemma(mo1)}We have
$$\sum\limits_{j=2}^{l}\|D_{T_j}^*h\|_2^2\leq\frac{\omega(1-\omega)}{s}\left(\sum\limits_{j=1}^{l}\|D_{T_j}^*h\|_1\right)^2.$$
\end{lemma}
\begin{proof}By the simple inequality
$$\sum\limits_{j=1}^{d}|x_j|^2\leq\max\limits_{1\leq j\leq
d}|x_j|\sum\limits_{j=1}^{d}|x_j|,$$ we have
$$\sum\limits_{j=2}^{l}\|D_{T_j}^*h\|_2^2\leq |x_{2s+1}|\sum\limits_{j=2}^{l}\|D_{T_j}^*h\|_1
\leq
\frac{\omega(1-\omega)}{s}\left(\sum\limits_{j=1}^{l}\|D_{T_j}^*h\|_1\right)^2.$$
\end{proof}

\begin{lemma}\label{lemma(mo2)}We have
$$\sum\limits_{j=2}^{l}\|D_{T_j}^*h\|_2^2
+\delta_{2s}\left(\sum\limits_{j=2}^{l}\|D_{T_j}^*h\|_2\right)^2\leq\frac{\omega(1-\omega)+\delta_{2s}\left(1-{3\omega}/{4}\right)^2}{s}
\left(\sum\limits_{j=1}^{l}\|D_{T_j}^*h\|_1\right)^2.$$
\end{lemma}
\begin{proof}By \cite[Proposition 1]{Cai}, we have
$$s^{1/2}\|D_{T_j}^*h\|_2\leq \|D_{T_j}^*h\|_1+s(|x_{js+1}|-|x_{js+s}|)/4,\quad
j=2,\cdots,l.$$ Therefore, we have \bea
s^{1/2}\sum\limits_{j=2}^{l}\|D_{T_j}^*h\|_2\leq
\sum\limits_{j=2}^{l}\|D_{T_j}^*h\|_1+s|x_{2s+1}|/4\leq
\sum\limits_{j=2}^{l}\|D_{T_j}^*h\|_1
+\|D_{T_1}^*h\|_1/4=(1-3\omega/4)\sum\limits_{j=1}^{l}\|D_{T_j}^*h\|_1.
\eea Combining the above inequality with Lemma \ref{lemma(mo1)}, one
can finish the proof.
\end{proof}

Since $\hat{f}$ is a minimizer of $(P_{1,\varepsilon})$, one gets
that
$$\|D^*f\|_1\geq\|D^*\hat{f}\|_1.$$ That is
$$\|D^*_Tf\|_1+\|D^*_{T^c}f\|_1\geq\|D^*_T\hat{f}\|_1+\|D^*_{T^c}\hat{f}\|_1.$$
Thus
$$\|D^*_Tf\|_1+\|D^*_{T^c}f\|_1\geq\|D^*_Tf\|_1-\|D^*_Th\|_1+\|D^*_{T^c}h\|_1-\|D^*_{T^c}f\|_1.$$
This implies \be\label{cone constraint1}
\sum\limits_{j=1}^{l}\|D^*_{T_j}h\|_1\leq2\|D^*_{T^c}f\|_1+\|D^*_Th\|_1.\ee
According to the feasibility of $\hat{f}$, $Ah$ must be small:
\be\label{Ahbound1}
\|Ah\|_2=\|Af-A\hat{f}\|_2\leq\|Af-y\|_2+\|A\hat{f}-y\|_2\leq2\varepsilon.\ee

\subsection{General signal recovery} For $\delta_{2s}<2/3$, denote that
\be\label{e11}\rho_s=\sqrt{\frac{4(1+5\delta_{2s}-4\delta_{2s}^2)}{(1-\delta_{2s})(32-25\delta_{2s})}}.\ee
By a easy computation, one can show that if
$\delta_{2s}<(77-\sqrt{1337})/82\approx0.4931$, then $\rho_s<1$.

\begin{lemma}\label{lemma6}If $\delta_{2s}<0.4931$, then
\be\label{e6}\sum\limits_{j=1}^{l}\|D_{T_j}^*h\|_1\leq
\frac{2}{1-\rho_s}\|D^*f-(D^*f)_{[s]}\|_1+\frac{2\sqrt{2}}{(1-\rho_s)\sqrt{1-\delta_{2s}}}\sqrt{s}\varepsilon.\ee
\end{lemma}
\begin{proof}By $Ah=\sum_{j=0}^{l}ADD_{T_j}^*h$ and (\ref{Ahbound1}), we have
\begin{eqnarray}0&=&\|Ah-\sum_{j=2}^{l}ADD_{T_j}^*h\|_2^2-\|ADD_{T_{01}}^*h\|_2^2\nonumber\\
&\leq&(2\varepsilon+\|\sum_{j=2}^{l}ADD_{T_j}^*h\|_2)^2-\|ADD_{T_{01}}^*h\|_2^2\nonumber\\
&=&4\varepsilon^2+4\varepsilon\|\sum_{j=2}^{l}ADD_{T_j}^*h\|_2+\|\sum_{j=2}^{l}ADD_{T_j}^*h\|_2^2-\|ADD_{T_{01}}^*h\|_2^2\nonumber
\end{eqnarray}
Applying lemmas \ref{lemma3} and \ref{lemma4} to the above
inequality yields
\be\label{e5}\|D_{T_{01}}^*h\|_2^2\leq\frac{(2\varepsilon+\mathcal{N})^2}{1-\delta_{2s}},\ee
where
$$\mathcal{N}=\sqrt{\sum\limits_{j=2}^{l}\|D_{T_j}^*h\|_2^2
+\delta_{2s}\left(\sum\limits_{j=2}^{l}\|D_{T_j}^*h\|_2\right)^2}.$$
Using
\be\label{e8}\|D_{T_{01}}^*h\|_2^2=\|D_{T_{0}}^*h\|_2^2+\|D_{T_{1}}^*h\|_2^2
\geq\frac{\|D_{T_{0}}^*h\|_1^2+\|D_{T_{1}}^*h\|_1^2}{s}\ee to
(\ref{e5}), one can get
\begin{eqnarray}\label{e12}\|D_{T_{0}}^*h\|_1^2&\leq&\frac{s(2\varepsilon+\mathcal{N})^2-(1-\delta_{2s})\|D_{T_{1}}^*h\|_1^2}{1-\delta_{2s}}\nonumber\\
&\leq&\frac{(2\sqrt{2s}\varepsilon)^2+2\cdot2\sqrt{2s}\varepsilon\cdot\sqrt{s/2}\mathcal{N}+s\mathcal{N}^2-(1-\delta_{2s})\|D_{T_{1}}^*h\|_1^2}{1-\delta_{2s}}.\end{eqnarray}
Notice that by Lemma \ref{lemma(mo2)}, we have \bea
\sqrt{\frac{s}{2}}\mathcal{N}
&\leq&\sqrt{\frac{\omega(1-\omega)+\delta_{2s}\left(1-{3\omega}/{4}\right)^2}{2}}\sum\limits_{j=1}^{l}\|D_{T_j}^*h\|_1\\
&\leq&\sqrt{\frac{4(1+5\delta_{2s}-4\delta_{2s}^2)}{(32-25\delta_{2s})}}\sum\limits_{j=1}^{l}\|D_{T_j}^*h\|_1
\eea and \bea s\mathcal{N}^2-(1-\delta_{2s})\|D_{T_{1}}^*h\|_1^2
&\leq&[\omega(1-\omega)+\delta_{2s}\left(1-{3\omega}/{4}\right)^2-(1-\delta_{2s})\omega^2]
\left(\sum\limits_{j=1}^{l}\|D_{T_j}^*h\|_1\right)^2\\
&\leq&{\frac{4(1+5\delta_{2s}-4\delta_{2s}^2)}{(32-25\delta_{2s})}}
\left(\sum\limits_{j=1}^{l}\|D_{T_j}^*h\|_1\right)^2,\eea where we
have used the fact for all $\delta_{2s}\in [0,2/3)$,
\bea\max_{\omega\in
[0,1]}\frac{\omega(1-\omega)+\delta_{2s}\left(1-{3\omega}/{4}\right)^2}{2}
&\leq&\frac{\omega(1-\omega)+\delta_{2s}\left(1-{3\omega}/{4}\right)^2}{2}\bigg\rvert_{\omega=\frac{4(2-3\delta_{2s})}{16-9\delta_{2s}}}\\
&=&\frac{2(1+\delta_{2s})}{16-9\delta_{2s}} \leq
\frac{4(1+5\delta_{2s}-4\delta_{2s}^2)}{(32-25\delta_{2s})}\eea and
\bea\max_{\omega\in
[0,1]}\omega(1-\omega)+\delta_{2s}\left(1-{3\omega}/{4}\right)^2-(1-\delta_{2s})\omega^2
\leq\frac{4(1+5\delta_{2s}-4\delta_{2s}^2)}{(32-25\delta_{2s})}.\eea
Thus, it follows from the above two inequalities and (\ref{e12})
that
$$\|D^*_Th\|_1\leq\frac{2\sqrt{2s}\varepsilon}{\sqrt{1-\delta_{2s}}}
+\rho_s \sum\limits_{j=1}^{l}\|D^*_{T_j}h\|_1.$$ Combining with
(\ref{cone constraint1}), one can finish the proof.\end{proof}

The main result of this subsection is the following theorem.
\begin{thm}\label{thm5}If $\delta_{2s}<0.4931$, then \be\label{result(1)}\|\hat{f}-f\|_2\leq C_0\frac{\|D^*f-(D^*f)_{[s]}\|_1}{\sqrt{s}}
+C_1\varepsilon,\ee where
$$C_0=\frac{4}{1-\rho_s}\sqrt{\frac{2(2-\delta_{2s})}{(1-\delta_{2s})(32-25\delta_{2s})}},\quad
C_1=\frac{2}{\sqrt{1-\delta_{2s}}}\left(1+\frac{C_0}{\sqrt{2}}\right)$$
and
$$\rho_s=\sqrt{\frac{4(1+5\delta_{2s}-4\delta_{2s}^2)}{(1-\delta_{2s})(32-25\delta_{2s})}}.$$
\end{thm}

\begin{proof}

By (\ref{e5}), we have
\bea&&\|h\|_2^2=\|D^*h\|_2^2=\|D_{T_{01}}^*h\|_2^2+\sum\limits_{j=2}^{l}\|D_{T_j}^*h\|_2^2\\
&\leq&\frac{1}{1-\delta_{2s}}\left(2\varepsilon+\sqrt{\sum\limits_{j=2}^{l}\|D_{T_j}^*h\|_2^2
+\delta_{2s}\left(\sum\limits_{j=2}^{l}\|D_{T_j}^*h\|_2\right)^2}\right)^2
+\sum\limits_{j=2}^{l}\|D_{T_j}^*h\|_2^2.\eea Hence, we have
\be\label{e16}\|h\|_2\leq\frac{2\varepsilon}{\sqrt{1-\delta_{2s}}}
+\sqrt{\frac{1}{(1-\delta_{2s})}\left[\sum\limits_{j=2}^{l}\|D_{T_j}^*h\|_2^2
+\delta_{2s}\left(\sum\limits_{j=2}^{l}\|D_{T_j}^*h\|_2\right)^2\right]+\sum\limits_{j=2}^{l}\|D_{T_j}^*h\|_2^2}.\ee
Then by lemmas \ref{lemma(mo1)} and \ref{lemma(mo2)}, we have
\bea&&\|h\|_2\leq\frac{2\varepsilon}{\sqrt{1-\delta_{2s}}}
+\frac{\sqrt{(2-\delta_{2s})\omega(1-\omega)+\delta_{2s}(1-3\omega/4)^2}}{\sqrt{s}\sqrt{(1-\delta_{2s})}}
\sum\limits_{j=1}^{l}\|D_{T_j}^*h\|_1.\eea A direct calculation
shows that
$$\max\limits_{0<\omega\leq1}(2-\delta_{2s})\omega(1-\omega)+\delta_{2s}(1-3\omega/4)^2\leq\frac{8(2-\delta_{2s})}{32-25\delta_{2s}}.$$
It follows from the above two inequalities that
$$\|h\|_2\leq\frac{2\varepsilon}{\sqrt{1-\delta_{2s}}}
+\frac{2}{\sqrt{s}}\sqrt{\frac{2(2-\delta_{2s})}{(1-\delta_{2s})(32-25\delta_{2s})}}
\sum\limits_{j=1}^{l}\|D_{T_j}^*h\|_1.$$ Therefore, with the above
inequality and Lemma \ref{lemma6} we prove the result.
\end{proof}
\subsection{Special case: $ n\leq 4s$ }
For $\delta_{2s}\in [0,1)$, denote that
\be\label{e14}\rho_s=\sqrt{\frac{(1+\delta_{2s})^2}{8(1-\delta_{2s})}}.\ee
By a easy computation, one can show that if
$\delta_{2s}<4\sqrt{2}-5\approx0.656,$ $\rho_s<1$.

We have $l\leq3$ by $n\leq 4s.$ For simplicity, we assume that
$l=3$. Instead of lemmas \ref{lemma3} and \ref{lemma4}, we have the
following results.
\begin{lemma}\label{lemma7}We have
$$\|ADD_{T_{23}}^*h\|_2^2\leq(1+\delta_{2s})\|DD_{T_{23}}^*h\|_2^2\leq (1+\delta_{2s})\|D_{T_{23}}^*h\|_2^2.$$
\end{lemma}
\begin{proof}The proof is straightforward.\end{proof}
\begin{lemma}\label{lemma8}We have
$$\|ADD_{T_{23}}^*h\|_2^2-\|ADD_{T_{01}}^*h\|_2^2\leq(1+\delta_{2s})\|D_{T_{23}}^*h\|_2^2-(1-\delta_{2s})\|D_{T_{01}}^*h\|_2^2.$$
\end{lemma}
\begin{proof}By the definition of $\delta_{2s}$, we have
\be \|ADD_{T_{23}}^*h\|_2^2-\|ADD_{T_{01}}^*h\|_2^2 \leq
(1+\delta_{2s})\|DD_{T_{23}}^*h\|_2^2-(1-\delta_{2s})\|DD_{T_{01}}^*h\|_2^2.\ee
Applying (\ref{e9}) to the above yields \bea
\|ADD_{T_{23}}^*h\|_2^2-\|ADD_{T_{01}}^*h\|_2^2
&\leq&(1+\delta_{2s})\|D_{T_{23}}^*h\|_2^2-(1+\delta_{2s})\|D_{T_{01}}^*h\|_2^2
+2\delta_{2s}\|DD_{T_{01}}^*h\|_2^2\\
&\leq&(1+\delta_{2s})\|D_{T_{23}}^*h\|_2^2-(1-\delta_{2s})\|D_{T_{01}}^*h\|_2^2.\eea
\end{proof}

Similar as Lemma \ref{lemma6}, we have the following result.
\begin{lemma}\label{lemma9}If $n\leq 4s$ and $\delta_{2s}<0.656$, then
\be\label{e15}\sum\limits_{j=1}^{3}\|D_{T_j}^*h\|_1\leq
\frac{2}{1-\rho_s}\|D^*f-(D^*f)_{[s]}\|_1+\frac{2\sqrt{2}}{(1-\rho_s)\sqrt{1-\delta_{2s}}}\sqrt{s}\varepsilon.\ee
\end{lemma}
\begin{proof}Applying lemmas
\ref{lemma7} and \ref{lemma8}, and a similar approach as that for
(\ref{e12}) yields \be\label{e13}\|D_{T_{0}}^*h\|_1^2
\leq\frac{(2\sqrt{2s}\varepsilon)^2+2\cdot2\sqrt{2s}\varepsilon\cdot\sqrt{s/2}\mathcal{N}+s\mathcal{N}^2-(1-\delta_{2s})\|D_{T_{1}}^*h\|_1^2}{1-\delta_{2s}},\ee
where $$\mathcal{N}=\sqrt{(1+\delta_{2s})}\|D_{T_{23}}^*h\|_2.$$
Notice that by applying Lemma \ref{lemma(mo1)}, we have
$$\sqrt{\frac{s}{2}}\mathcal{N}
\leq\sqrt{\frac{\omega(1-\omega)(1+\delta_{2s})}{2}}\sum\limits_{j=1}^{3}\|D_{T_j}^*h\|_1
\leq\sqrt{\frac{1+\delta_{2s}}{8}}\sum\limits_{j=1}^{3}\|D_{T_j}^*h\|_1
\leq\sqrt{\frac{(1+\delta_{2s})^2}{8}}\sum\limits_{j=1}^{3}\|D_{T_j}^*h\|_1$$
and \bea s\mathcal{N}^2-(1-\delta_{2s})\|D_{T_{1}}^*h\|_1^2
&\leq&[\omega(1-\omega)(1+\delta_{2s})-(1-\delta_{2s})\omega^2]
\left(\sum\limits_{j=1}^{3}\|D_{T_j}^*h\|_1\right)^2\\
&\leq&\frac{(1+\delta_{2s})^2}{8}
\left(\sum\limits_{j=1}^{3}\|D_{T_j}^*h\|_1\right)^2.\eea Thus, it
follows from the above two inequalities and (\ref{e13}) that
$$\|D^*_Th\|_1\leq\frac{2\sqrt{2s}\varepsilon}{\sqrt{1-\delta_{2s}}}
+\rho_s \sum\limits_{j=1}^{3}\|D^*_{T_j}h\|_1.$$ Combining with
(\ref{cone constraint1}), one can finish the proof.
\end{proof}
The main result of this subsection is the following theorem.
\begin{thm}\label{thm6}If $n\leq 4s$ and $\delta_{2s}<0.656$, then \be\label{result(1)}\|\hat{f}-f\|_2\leq
C_0\frac{\|D^*f-(D^*f)_{[s]}\|_1}{\sqrt{s}}+C_1\varepsilon,\ee where
$$C_0=\frac{\sqrt{2}}{(1-\rho_s)\sqrt{(1-\delta_{2s})}},\quad
C_1=\frac{2}{\sqrt{1-\delta_{2s}}}\left(1+\frac{C_0}{\sqrt{2}}\right)
\quad\mbox{and}\quad\rho_s=\sqrt{\frac{(1+\delta_{2s})^2}{8(1-\delta_{2s})}}.$$
\end{thm}

\begin{proof}
By a similar approach as that for (\ref{e16}), we have
\bea&&\|h\|_2\leq\frac{2\varepsilon}{\sqrt{1-\delta_{2s}}}
+\sqrt{\frac{1}{(1-\delta_{2s})}(1+\delta_{2s})\|D_{T_{23}}^*h\|_2^2+\|D_{T_{23}}^*h\|_2^2}.\eea
Then by lemmas \ref{lemma(mo1)}, we have
\bea&&\|h\|_2\leq\frac{2\varepsilon}{\sqrt{1-\delta_{2s}}}
+\frac{\sqrt{2\omega(1-\omega)}}{\sqrt{s}\sqrt{(1-\delta_{2s})}}
\sum\limits_{j=1}^{3}\|D_{T_j}^*h\|_1\leq
\frac{2\varepsilon}{\sqrt{1-\delta_{2s}}}
+\frac{1}{\sqrt{s}\sqrt{2(1-\delta_{2s})}}
\sum\limits_{j=1}^{3}\|D_{T_j}^*h\|_1.\eea Therefore, with the above
inequality and Lemma \ref{lemma9} we prove the result.
\end{proof}
\begin{rem}When $D$ is an identity matrix, that is for the classical RIP and $l_1$-minimization $(L_{1,\varepsilon})$,
theorems \ref{thm5} and \ref{thm6} were proved in
\cite{Mo}.\end{rem}

\section{Recovery via $l_q$-minimization with $0< q< 1$}
In this section, we will discuss recovery of a signal by
$l_q$-minimization $(P_{q,\varepsilon})$ with $0<q<1$. For
$q\in(0,1]$, let $h=\hat{f}-f$, where $\hat{f}$ is the solution of
$(P_{q,\varepsilon})$ and $f$ is the original signal. We follow the
same assumptions as in Section 2. Moreover, rearranging the indices
if necessary, we assume that the first $s$ coordinates of $D^*f$ are
the largest in magnitude and
$|x_{s+1}|\geq|x|_{s+2}\geq\cdots\geq|x_d|.$ For the rest of this
section, for $q\in(0,1]$, we will always assume that
$\|D_{T_1}^*h\|_q^q=\omega\sum_{j=1}^{l}\|D_{T_j}^*h\|_q^q$ for some
nonnegative number $\omega=\omega(q)\in[0,1]$. Then we have
$\sum_{j=2}^{l}\|D_{T_j}^*h\|_q^q=(1-\omega)\sum_{j=1}^{l}\|D_{T_j}^*h\|_q^q$.

\begin{lemma}\label{lemma(lai1)}For $q\in(0,1]$, we have
$$\sum\limits_{j=2}^{l}\|D_{T_j}^*h\|_2^2\leq\frac{(1-\omega)\omega^{(2-q)/q}}{s^{(2-q)/q}}\left(\sum\limits_{j=1}^{l}\|D_{T_j}^*h\|_q^q\right)^{2/q}.$$
\end{lemma}
\begin{proof}By the simple inequality
$$\sum\limits_{j=1}^{d}|x_j|^2\leq\max\limits_{1\leq j\leq
d}|x_j|^{2-q}\sum\limits_{j=1}^{d}|x_j|^{q},$$ we have
\bea\sum\limits_{j=2}^{l}\|D_{T_j}^*h\|_2^2&\leq&
|x_{2s+1}|^{2-q}\sum\limits_{j=2}^{l}\|D_{T_j}^*h\|_q^q \leq
\left(\frac{\|D_{T_1}^*h\|_q^q}{s}\right)^{(2-q)/q}\sum\limits_{j=2}^{l}\|D_{T_j}^*h\|_q^q\\
&=&\frac{(1-\omega)\omega^{(2-q)/q}}{s^{(2-q)/q}}\left(\sum\limits_{j=1}^{l}\|D_{T_j}^*h\|_q^q\right)^{2/q}.\eea
\end{proof}

\begin{lemma}\label{lemma(lai2)}For $q\in(0,1]$, we have
$$\sum\limits_{j=2}^{l}\|D_{T_j}^*h\|_2^2
+\delta_{2s}\left(\sum\limits_{j=2}^{l}\|D_{T_j}^*h\|_2\right)^2\leq\frac{(1-\omega)\omega^{(2-q)/q}+\delta_{2s}}{{s}^{2/q-1}}\left(\sum\limits_{j=1}^{l}\|D_{T_j}^*h\|_q^q\right)^{2/q}.$$
\end{lemma}
\begin{proof}By \cite[Lemma 2]{Lai}, we have
$$s^{1/q-1/2}\|D_{T_j}^*h\|_2\leq \|D_{T_j}^*h\|_q+s^{1/q}(|x_{js+1}|-|x_{js+s}|),\quad
j=2,\cdots,l.$$ Therefore, we have \bea
s^{1/q-1/2}\sum\limits_{j=2}^{l}\|D_{T_j}^*h\|_2\leq
\sum\limits_{j=2}^{l}\|D_{T_j}^*h\|_q+s^{1/q}|x_{2s+1}|\leq
\sum\limits_{j=2}^{l}\|D_{T_j}^*h\|_q
+\|D_{T_1}^*h\|_q\leq\left(\sum\limits_{j=1}^{l}\|D_{T_j}^*h\|_q^q\right)^{1/q},\eea
where for the last inequality we have used $\|n\|_1\leq\|n\|_q$ for
$n\in\mathbb{R}^l$. Combining with Lemma \ref{lemma(lai1)}, one can
conclude the result.
\end{proof}
Analogous to (\ref{cone constraint1}) and (\ref{Ahbound1}), one can
prove that \be\label{cone constraint2}
\sum\limits_{j=1}^{l}\|D^*_{T_j}h\|_q^q\leq2\|D^*_{T^c}f\|_q^q+\|D^*_Th\|_q^q\ee
and \be\label{Ahbound2} \|Ah\|_2\leq2\varepsilon.\ee

Denote that
$$\rho_s(q)=\sqrt{\frac{\delta_{2s}}{1-\delta_{2s}}+\frac{q}{2^{2/q}(1-\delta_{2s})}
\left(\frac{2-q}{2-\delta_{2s}}\right)^{\frac{2}{q}-1}}.$$ For
$\delta_{2s}<\frac{1}{2}$, one can prove that there exists a value
$q_0=q_0(\delta_{2s})\in(0,1]$ such that for all $q\in(0,q_0)$,
$\rho_s(q)<1$. Indeed, by a easy calculation, $\rho_s(q)<1$ is
equivalent to $$\delta_{2s}+\frac{q}{2^{2/q+1}}
\left(\frac{2-q}{2-\delta_{2s}}\right)^{\frac{2}{q}-1}<\frac{1}{2}.$$
Since the second term on the left hand side goes to zero as
$q\rightarrow 0_+$ as $\delta_{2s}<1,q\leq1$ and
$$\frac{1}{2^{2/q+1}}
\left(\frac{2-q}{2-\delta_{2s}}\right)^{\frac{2}{q}-1}\leq
\left(\frac{2-q}{2}\right)^{2/q}\approx\frac{1}{e},$$ one can finish
the conclusion.

\begin{lemma}\label{lemma10}If $\delta_{2s}<1/2$ and $q\in(0,q_0)$, then
$$\left(\sum\limits_{j=1}^{l}\|D_{T_j}^*h\|_q^q\right)^{\frac{1}{q}}
\leq\frac{2^{\frac{2}{q}-1}}{{(1-\rho^q_s(q))}^{1/q}}\|D^*f-(D^*f)_{[s]}\|_q
+\frac{2^{\frac{2}{q}-\frac{1}{2}}s^{\frac{1}{q}-\frac{1}{2}}\varepsilon}{{(1-\rho^q_s(q))}^{1/q}\sqrt{(1-\delta_{2s})}}.$$
\end{lemma}
\begin{proof}Notice that we also have (\ref{e5}).
Using
\be\label{e17}\|D_{T_{01}}^*h\|_2^2=\|D_{T_{0}}^*h\|_2^2+\|D_{T_{1}}^*h\|_2^2
\geq\frac{\|D_{T_{0}}^*h\|_q^2+\|D_{T_{1}}^*h\|_q^2}{s^{(2-q)/q}}\ee
to (\ref{e5}), one can get
\begin{eqnarray}\label{e18}\|D_{T_{0}}^*h\|_q^2&\leq&\frac{s^{(2-q)/q}(2\varepsilon+\mathcal{N})^2-(1-\delta_{2s})\|D_{T_{1}}^*h\|_q^2}{1-\delta_{2s}}\nonumber\\
&\leq&\frac{(2(2s)^{\frac{1}{q}-\frac{1}{2}}\varepsilon)^2+2\cdot2(2s)^{\frac{1}{q}-\frac{1}{2}}\varepsilon\cdot
\left(\frac{s}{2}\right)^{\frac{1}{q}-\frac{1}{2}}\mathcal{N}+s^{\frac{2}{q}-1}\mathcal{N}^2-(1-\delta_{2s})\|D_{T_{q}}^*h\|_q^2}{1-\delta_{2s}}.\end{eqnarray}
Notice that by Lemma \ref{lemma(lai2)}, we have \bea
\left(\frac{s}{2}\right)^{\frac{1}{q}-\frac{1}{2}}\mathcal{N}
&\leq&\sqrt{\frac{(1-\omega)\omega^{(2-q)/q}+\delta_{2s}}{2^{(2-q)/q}}}
\left(\sum\limits_{j=1}^{l}\|D_{T_j}^*h\|_q^q\right)^{\frac{1}{q}}\\
&\leq&\
\sqrt{\delta_{2s}+\frac{q}{2^{2/q}}\left(\frac{2-q}{2-\delta_{2s}}\right)^{\frac{2}{q}-1}}
\left(\sum\limits_{j=1}^{l}\|D_{T_j}^*h\|_q^q\right)^{\frac{1}{q}}
\eea and \bea
s^{\frac{2}{q}-1}\mathcal{N}^2-(1-\delta_{2s})\|D_{T_{1}}^*h\|_q^2
&\leq&[\omega^{(2-q)/q}(1-\omega)+\delta_{2s}-(1-\delta_{2s})\omega^{2/q}]
\left(\sum\limits_{j=1}^{l}\|D_{T_j}^*h\|_q^q\right)^{\frac{2}{q}}\\
&\leq&\left(\delta_{2s}+\frac{q}{2^{2/q}}\left(\frac{2-q}{2-\delta_{2s}}\right)^{\frac{2}{q}-1}\right)
\left(\sum\limits_{j=1}^{l}\|D_{T_j}^*h\|_q^q\right)^{\frac{2}{q}},\eea
where we have used the fact for all $\delta_{2s}\in [0,1)$,
\bea\max_{\omega\in
[0,1]}\frac{(1-\omega)\omega^{(2-q)/q}+\delta_{2s}}{2^{(2-q)/q}}
&\leq&\frac{(1-\omega)\omega^{(2-q)/q}+\delta_{2s}}{2^{(2-q)/q}}\Big\rvert_{\omega=\frac{2-q}{2}}
=\frac{\delta_{2s}+\frac{q}{2^{2/q}}\left(2-q\right)^{\frac{2}{q}-1}}{2^{(2-q)/q}}\\
&\leq&
\delta_{2s}+\frac{q}{2^{2/q}}\left(\frac{2-q}{2-\delta_{2s}}\right)^{\frac{2}{q}-1}\eea
and \bea&&\max_{\omega\in
[0,1]}\omega^{(2-q)/q}(1-\omega)+\delta_{2s}-(1-\delta_{2s})\omega^{2/q}\\
&\leq&\omega^{(2-q)/q}(1-\omega)+\delta_{2s}-(1-\delta_{2s})\omega^{2/q}\big\rvert_{\omega=\frac{2-q}{2(2-\delta_{2s})}}
=\delta_{2s}+\frac{q}{2^{2/q}}\left(\frac{2-q}{2-\delta_{2s}}\right)^{\frac{2}{q}-1}.\eea
Thus, it follows from the above two inequalities and (\ref{e18})
that
$$\|D^*_Th\|_q\leq\frac{2(2s)^{\frac{1}{q}-\frac{1}{2}}\varepsilon}{\sqrt{1-\delta_{2s}}}
+\rho_s(q)
\left(\sum\limits_{j=1}^{l}\|D^*_{T_j}h\|_q^q\right)^{\frac{1}{q}}.$$
Therefore, we get
$$\|D^*_Th\|_q^q\leq\frac{2^q(2s)^{1-\frac{q}{2}}\varepsilon^q}{(1-\delta_{2s})^{\frac{q}{2}}}
+\rho_s^q(q) \sum\limits_{j=1}^{l}\|D^*_{T_j}h\|_q^q,$$ where we
have used $(|a|+|b|)^q\leq |a|^q+|b|^q.$ Plug the above inequality
to (\ref{cone constraint2}) and by a easy computation, one can show
that
$$\sum\limits_{j=1}^{l}\|D_{T_j}^*h\|_q^q\leq
\frac{2}{1-\rho_s^q(q)}\|D_{T^c}^*f\|_q^q
+\frac{2^q(2s)^{1-\frac{q}{2}}\varepsilon^q}{(1-\delta_{2s})^{\frac{q}{2}}(1-\rho_s^q(q))}.$$
Using the basic inequality $\|v\|_q\leq 2^{\frac{1}{q}-1}\|v\|_1$
for $v\in\mathbb{R}^2$ to the above inequality, one can finish the
proof.
\end{proof}

The main result of this subsection is the following theorem.

\begin{thm}\label{thm7}If $\delta_{2s}<\frac{1}{2}$, then there exists a value $q_0=q_0(\delta_{2s})\in(0,1]$ such that for any $q\in(0,q_0)$,
$$\rho_s(q)=\sqrt{\frac{\delta_{2s}}{1-\delta_{2s}}+\frac{q}{2^{2/q}(1-\delta_{2s})}
\left(\frac{2-q}{2-\delta_{2s}}\right)^{\frac{2}{q}-1}}<1.$$
Furthermore, \be\|\hat{f}-f\|_2\leq
C_0\frac{\|D^*f-(D^*f)_{[s]}\|_q}{s^{\frac{1}{q}-\frac{1}{2}}}+
C_1\varepsilon,\ee where
$$C_0=\frac{2^{1/q-1}}{(1-\rho^q_s(q))^{1/q}}\sqrt{\frac{(2-\delta_{2s})(2-q)^{(2-q)/q}q+2^{2/q}\delta_{2s}}{1-\delta_{2s}}}\quad
\mbox{and}\quad
C_1=\frac{2}{\sqrt{1-\delta_{2s}}}\left(1+\frac{C_0}{\sqrt{2}}\right).$$\end{thm}

\begin{proof}Notice that we also have (\ref{e16}). Using lemmas \ref{lemma(lai1)} and \ref{lemma(lai2)} to
(\ref{e16}), we get
\bea\|h\|_2&\leq&\frac{2\varepsilon}{\sqrt{1-\delta_{2s}}}
+\sqrt{\frac{(2-\delta_{2s})\omega^{(2-q)/q}(1-\omega)+\delta_{2s}}{(1-\delta_{2s})s^{(2-q)/q}}}
\left(\sum\limits_{j=1}^{l}\|D_{T_j}^*h\|_q^q\right)^{\frac{1}{q}}\\
&\leq&\frac{2\varepsilon}{\sqrt{1-\delta_{2s}}}
+\sqrt{\frac{(2-\delta_{2s})(2-q)^{(2-q)/q}q+2^{2/q}\delta_{2s}}{(1-\delta_{2s})2^{2/q}s^{(2-q)/q}}}
\left(\sum\limits_{j=1}^{l}\|D_{T_j}^*h\|_q^q\right)^{\frac{1}{q}},\eea
where we have used $$\max_{\omega\in
[0,1]}(2-\delta_{2s})(1-\omega)\omega^{(2-q)/q}
\leq(2-\delta_{2s})(1-\omega)\omega^{(2-q)/q}\big\rvert_{\omega=\frac{2-q}{2}}.$$
Now the proof can be finished by directly applying Lemma
\ref{lemma10}.
\end{proof}
\begin{rem}When $D$ is an identity matrix, that is for the classical RIP and $l_q$-minimization $(L_{q,\varepsilon})$,
Theorem \ref{thm7} was shown in \cite{Lai}.\end{rem}

\end{document}